\documentclass[11pt,reqno]{article}

\usepackage{graphicx}
\usepackage{amsmath}
\usepackage{amsthm}
\usepackage{latexsym,bm}
\usepackage{amssymb}
\usepackage{indentfirst}

\usepackage{color,xcolor}

\newtheorem{thm}{Theorem}[section]
\newtheorem{cor}[thm]{Corollary}
\newtheorem{lem}[thm]{Lemma}

\newtheorem{exm}{Example}
\newtheorem{Remark}{Remark}
\numberwithin{equation}{section}

\newcommand{\C}{\mathbb{C}}
\newcommand{\R}{\mathbb{R}}
\newcommand{\Q}{\mathbb{Q}}
\newcommand{\CC}{\mathbb{C}}

\newcommand{\Z}{\mathbb{Z}}
\DeclareMathOperator{\ind}{Ind}
\DeclareMathOperator{\tr}{Tr}
\DeclareMathOperator{\cok}{coker}

\newcommand {\be}{\begin{equation}}
\newcommand {\ee}{\end{equation}}

\newcommand {\divg}{\mathrm{div}\,}
\newcommand {\llag}{\left\langle}
\newcommand {\rrag}{\right\rangle}

\begin{document}

\title{New Bochner type theorems}

\author{Xiaoyang Chen\footnotemark, Fei Han \footnotemark}
\renewcommand{\thefootnote}{\fnsymbol{footnote}}
\footnotetext{School of Mathematical Sciences, Institute for Advanced Study, Tongji University, Shanghai, China. email: xychen100@tongji.edu.cn}
\footnotetext{Department of Mathematics, National University of Singapore, Singapore. email: mathanf@nus.edu.sg }
 \date{}
\maketitle

\begin{abstract}
A classical theorem of Bochner asserts that the isometry group of a  compact Riemannian manifold with negative Ricci curvature is finite.
In this paper we give several extensions of Bochner's theorem by allowing ``small" positive Ricci curvature.

\end{abstract}


\section{Introduction}

A classical theorem of Bochner asserts that the isometry group of a compact Riemannian manifold with negative Ricci curvature is finite  \cite{Boc}.
In this paper we give several extensions of Bochner's theorem by allowing ``small" positive Ricci curvature.
Our first result is the following:

\begin{thm} \label{mm0}
Let $M$ be a compact $2n$-dimensional complex manifold with  nonzero holomorphic Euler number.
Then given a positive number $\lambda_1$, there exists some $\epsilon=\epsilon(n, \lambda_1)>0$ such that
the isometry group of any K\"ahler metric $g$ on $M$ is finite provided that
$$-\lambda_1 \leq Ric(g) \leq \epsilon, \ diam(g) \leq 1.$$

\end{thm}

We also have the following Riemannian analogue of Theorem \ref{mm0} under an additional integral curvature bound.
\begin{thm} \label{mm1}
Let $M$ be a compact $n$-dimensional smooth manifold with nonzero Euler number or nonzero signature.
Then given positive numbers $p, \lambda_1, \lambda_2$ with $p>n/2$, there exists some $\epsilon=\epsilon(p, n, \lambda_1, \lambda_2)>0$ such that
the isometry group of any Riemannian metric $g$ on $M$ is finite provided that:
$$-\lambda_1 \leq Ric(g) \leq \epsilon, diam (g)\leq 1, \ \frac{1}{V(g)}\int_M |\Re_g|^p dV
\leq \lambda_2.$$
\end{thm}

In the spin case, we have the following extension of Bochner's theorem for $4n$ dimensional manifolds involving the topological invariants: elliptic genera. See Section 5 for an Appendix about a brief introduction to elliptic genera.
\begin{thm} \label{mm2}
Let $M$ be a compact $4n$-dimensional spin manifold with nonzero elliptic genera. Then given  positive numbers $p, \lambda_1, \lambda_2$ with $p>2n$,
there exists some $\epsilon=\epsilon(p, n, \lambda_1, \lambda_2)>0$ such that
the isometry group of any Riemannian metric $g$ on $M$ is finite provided that:
$$-\lambda_1 \leq Ric(g) \leq \epsilon, diam (g)\leq 1, \frac{1}{V(g)}\int_M |\Re_g|^p dV
\leq \lambda_2.$$
\end{thm}

In Theorems \ref{mm0}-\ref{mm2}, $Ric(g)$ / $\Re_g$/ $V(g)$/ $diam(g)$ stands for the Ricci curvature/ Riemannian curvature tensor/ volume/ diameter of $g$, respectively.

\begin{Remark}
 The topological assumptions in  Theorems \ref{mm0}-\ref{mm2} are indispensable. For example, the $4n$-dimensional torus $T^{4n}$ has vanishing holomorphic Euler number, Euler number, signature and elliptic genera. However, $T^{4n}$  admits a flat metric with infinite isometry group.
   In fact, the main contribution of this paper is to extend  Bochner's theorem to Riemannian manifolds which may have small positive Ricci curvature
   under certain topological assumptions.
\end{Remark}

\begin{Remark}
There exists a compact $4n$-dimensional spin manifold with infinite isometry group and nonzero
elliptic genera. See Example \ref{va} in Section 5 for details. This shows that the curvature assumption in Theorem \ref{mm2} is necessary.
\end{Remark}

 The famous Atiyah-Hirzebruch vanishing theorem \cite{AH70}  asserts that if a $4n$ dimensional spin manifold $M$ admits a nontrivial smooth circle action, then the $\hat A$-genus of $M$ vanishes. If $M$ is a $K3$ surface, then it is known that $\hat A(M)=2$. This shows that $K3$ surfaces do not admit nontrivial smooth circle actions. Yau conjectured that a simply connected compact Calabi Yau manifold does not admit a smooth circle action. As an application of our Theorem \ref{mm0}, we give a partial confirmation to Yau's conjecture in the following Corollary:
 \begin{cor} \label{cr1}
Let $(M,g)$ be a compact Calabi-Yau manifold with nonzero holomorphic Euler number,
then $g$ or its sufficiently small K\"ahler perturbation does not admit an isometric circle action.
Note that the topological assumption forces that $M$ has finite fundamental group by Cheeger-Gromoll splitting theorem \cite{CG}.
 \end{cor}

Theorem \ref{mm2} shows that a $4n$-dimensional closed spin manifold $(M,g)$ with infinite isometry group must have vanishing elliptic genera under approriate curvature assumptions. To the author's best knowledge, it is the first time in the literature that elliptic genera are shown to vanish under curvature assumptions. This vanishing phenomenon has an interesting application to string geometry. It is known that a spin manifold $M$ has a canonical spin class $\frac{p_1}{2}\in  H^4(M, \Z)$ determined by its spin structure, twice of which is the first Pontryagin class \cite{FH19}. $M$ is called string if the spin class $\frac{p_1}{2}=0$. Let $sec(g)$ be the sectional curvature of $g$. Applying Theorem \ref{mm2}, in \cite{HH21}, it can be shown that:

\begin{cor}\label{riccicoro}Given positive number $\lambda$, there exists some $\varepsilon =\varepsilon(\lambda) > 0$ such that
if a compact 24-dimensional string Rimannian manifold $(M, g)$ satisfies ${\rm diam}(g) \leq 1, {\rm Ric}(g) \leq \varepsilon,$
 $ sec(g) \geq -\lambda$ and has infinite isometry group, then $M$ bounds a string manifold.
\end{cor}
We would like to remark that  Corollary \ref{riccicoro} is related to a famous conjecture of Farrell-Zdravkovska \cite{FZ83} and Yau \cite{Yau93} saying that every almost flat manifold is the boundary of a closed manifold. Davis and Fang \cite{DF16} showed that this conjecture holds under the assumption that the $2$-sylow subgroup of holonomy group is cyclic or generalized quaternionic. The general case of the conjecture remains open. It is also pointed in \cite{DF16} that it is a difficult question that if every almost flat spin manifold (up to changing spin structures) bounds a spin manifold. Corollary \ref{riccicoro} asserts that under weaker curvature condition (weaker than almost flat), in dimension 24 (one of the most important dimensions for string geometry), every string manifold (up to changing string structures) bounds a string manifold.

 $\, $

The geometry and topology of Riemannian manifolds with bounded diameter and certain curvature bound have been
studied extensively. Let us only mention a few closely related to our work. In \cite{KN}, A. Katsuda and T. Nakamura
proved a rigidity theorem for Killing vector fields of a manifold with almost nonpositive Ricci curvature.
 Quantitative versions of Bochner's theorem were obtained in \cite{DSW, Kat, Lim}, more precisely, the authors obtained estimate of
the order of the isometry groups of compact manifolds in terms of some geometric data.

The proof of the results in \cite{DSW, Lim} were based on the collapsing theory of Riemannian manifolds.
Nevertheless Theorems \ref{mm0}-\ref{mm2} will be proved based on Bochner technique and the proof does not involve collapsing theory of Riemannian manifolds.

\par
In Theorem \ref{mm1} and \ref{mm2}, integral curvature bounds are involved. Actually integral curvature bounds have also been significantly used  in various geometric situations, such as the $L^2$-bound of the curvature tensor for noncollapsed manifolds with bounded Ricci curvature, and the $L^4$-bound of the Ricci curvature for the K\"ahler-Ricci flow as well as the real Ricci flow under certain conditions \cite{Ba, BaZ, CN, JN, Si, TiZ}. Recently, V. Kapovitch and J. Lott studied almost Ricci flat manifolds under certain
integral bound of the Riemannian curvature tensor \cite{KL}.
\par

The strategy to prove Theorems \ref{mm0}-\ref{mm2} is by contradiction based on Bochner technique  in a similar fashion.
 Suppose Theorem \ref{mm0} is not true, then given a positive number $\lambda_1$, there is a sequence of K\"ahler metrics $g_i$ on a
 compact complex manifold $M$ with nonzero holomorphic Euler number such that
 the isometry groups of $(M, g_i)$ are infinite and
 $$-\lambda_1 \leq Ric (g_i) \leq \frac{1}{i}, \ diam (g_i) \leq 1.$$
Recall that the holomorphic Euler number of $M$ is equal to $\sum_p (-1)^p dim H^{p,0}(M)$, which is also the index of the Dirac operator $P_i=\bar{\partial_i}+\bar{\partial_i}^*:
\oplus_p \wedge^{2p,0}(M) \rightarrow \oplus_p \wedge^{2p+1,0}(M)$. Then we see that
the index of $P_i$ is nonzero.

\par
The proof will consist of three main steps.

\par
Step 1: In this step the rigidity property of certain elliptic operators will be crucial for us.
We briefly recall the definition of rigidity of an elliptic operator.
 Let $M$ be a closed smooth manifold and
$P$ be an elliptic operator on $M$.
We assume that a compact connected Lie group $G$ acts on $M$ nontrivially and
that $P$ commutes with the $G$-action.
Then the kernel and cokernel of $P$ are finite dimensional representations of $G$.
The equivariant index of $P$ is the character of the virtual representation of $G$ defined by
\begin{equation}
\ind(P, h)=\tr\big[h\big|_{\ker P}\big]-\tr\big[h\big|_{\cok P}\big],\quad h\in G.
\end{equation}
$P$ is said to be \emph{rigid} for the $G$-action if $\ind(P, h)$ does not
depend on $h\in G$.

\par In the setting of Theorem \ref{mm0}, as the isometry group of $(M, g_i)$ is infinite,
there is a nonzero Killing vector field $X_i$ on $M$ generating an isometric $S^1$ action.
The following classical fact is crucial for us.

\begin{thm} \label{ri}
The Dirac operators $P_i=\bar{\partial_i}+\bar{\partial_i}^*$ is rigid for the isometric $S^1$ action.
\end{thm}

\begin{proof}
For each $i$, the Killing vector field $X_i$ is also holomorphic since $g_i$ is a K\"ahler metric on $M$ \cite{M}.
Then  $P_i$ commutes with the isometric $S^1$-action.
 Denote by $c_{p,0} (t)$ the trace of the automorphism of $H^{p,0} (M)$ induced by the map  $f_t =exp (t X_i): M \rightarrow M$.
 By Hodge theory, we see that $H^{p,0} (M)$ is a part of the de Rham cohomology group $H^p(M)$.
 As $f_t$ is homotopic to the identity map,  then $f_t$ induces a trivial action on  $H^{p,0} (M)$. Hence
  the sums  $\sum_p (-1)^p c_{p,0} (t)$ are independent of $t$.
 By Hodge theory, $H^{p,0} (M)$ is isomorphic to the kernel of $P_i$. Since $P_i$ is self dual, we see that
 $P_i=\bar{\partial_i}+\bar{\partial_i}^*$ is rigid for the isometric $S^1$ action.
\end{proof}

By Theorem \ref{ri}, we have
\be \ind P_i= \ind (P_i, 1)=\ind (P_i, \lambda), \ \ \forall \lambda\in S^1.\ee
As the equivariant index $\ind (P_i, \lambda)$  is a Laurent polynomial of $\lambda$ and independent on $\lambda\in S^1$, one must have
\be
\begin{split}
\ind P_i=\ind (P_i, \lambda)=&\dim \left(\ker P_i \cap \Gamma(\oplus_p \wedge^{2p,0}(M))^{S^1}\right)\\
&-\dim \left(\mathrm{coker} P_i \cap \Gamma(\oplus_p \wedge^{2p,0}(M))^{S^1}\right),
\end{split}
 \ee
 where $\Gamma(\oplus_p \wedge^{2p,0}(M))^{S^1}$ consists of smooth sections of $\oplus_p \wedge^{2p,0}(M)$
invariant under the $S^1$-action.
\par
Consider the following Witten deformation of $P_i$:
\be \label{deform} \widetilde{P_i}=P_i + \sqrt{-1} t_i c(X_i),\ee
where $t_i:=(\frac{V(g_i)}{\int_{M} |X_i|^2 dV_i})^{1/2}>0$ and $c(X_i)$ is the Clifford product.
Then $\widetilde{P_i}$ commutes with the isometric $S^1$-action. By restricting to $\Gamma(\oplus_p \wedge^{2p,0}(M))^{S^1}$
and homotopy invariance of index, we have
\be
\begin{split}
\dim \left(\ker P_i \cap \Gamma(\oplus_p \wedge^{2p,0}(M))^{S^1}\right)
-\dim \left(\mathrm{coker} P_i \cap \Gamma(\oplus_p \wedge^{2p,0}(M))^{S^1}\right)
\\
=\dim \left(\ker  \widetilde{P_i} \cap \Gamma(\oplus_p \wedge^{2p,0}(M))^{S^1}\right)
-\dim \left(\mathrm{coker} \widetilde{P_i} \cap \oplus_p \wedge^{2p,0}(M))^{S^1}\right).
\end{split}
 \ee

 Since  $\ind P_i \neq 0$ and  $\widetilde{P_i}$ is self adjoint, we see that there must exist
some $s_i \in \Gamma(\oplus_p \wedge^{2p,0}(M))$ or $\Gamma(\oplus_p \wedge^{2p+1,0}(M))$  such that
\begin{equation} \label{s}
s_i \neq 0
\end{equation}

\begin{equation}
\widetilde{P_i} s_i =0
\end{equation}

\begin{equation}
L_{X_i} s_i=0,
\end{equation}
where $L_{X_i} s_i $ is the Lie derivative of $s_i$ in the direction $X_i$.

\par

We will prove the following crucial inequality in section 2:
\begin{equation} \label{1111}
\int_{M}  t_i^2 |X_i|^2 |s_i|^2 dV_i \leq C(n) \int_{M} t_i |\nabla X_i| |s_i|^2 dV_i,
\end{equation}
where $C(n)$ is some constant depending only on $n$.
\\
\par Step 2: Applying Bochner formula to $X_i, s_i$, we will prove the following mean value inequality in section 3:
\begin{equation} \label{2111}
|X_i|^2_{\infty}=:max_{x \in M}|X_i|^2(x) \leq C_1 \frac{\int_{M} |X_i|^2 dV_i}{V(g_i)},
\end{equation}
\begin{equation} \label{2122}
|s_i|^2_{\infty}=:max_{x \in M}|s_i|^2(x) \leq C_2 \frac{\int_{M} |s_i|^2 dV_i}{V(g_i)},
\end{equation}
where $C_1, C_2$ are positive constants depending on  $n, \lambda_1$.
\begin{Remark}
\par In this step?? the curvature assumption that $ Ric (g_i)\geq -\lambda_1, diam (g_i) \leq 1$ will be used.
\end{Remark}

\par Step 3: As $Ric(g_i) \leq \frac{1}{i}$, applying Bochner formula to $X_i$, we get
\begin{equation} \label{2}
\int_{M} |\nabla X_i|^2 dV_i \leq \frac{1}{i} \int_{M} |X_i|^2 dV_i.
\end{equation}
 Combined with inequalities \ref{1111}, \ref{2111} and \ref{2122}, for sufficiently large $i$, in section 4 we will show that
 $$ \int_{M} |s_i|^2 dV_i \leq  \frac{1}{2}\int_{M} |s_i|^2 dV_i.$$
Hence $s_i \equiv 0$, which contradicts with \ref{s}.

\par To prove Theorem \ref{mm1}, we can consider the operators $d+d^*$ (acting on different Clifford bundles for Euler number and signature) and their Witten deformations. The operator $d+d^*$ is also rigid due to the homotopy equivalence of the de Rham cohomology group $H^* (M, \mathbb{R})$ on which $S^1$ always induces a trivial action.

The rest part of proof is almost identical as Theorem \ref{mm0}, except that in Step 2 we need  an additional integral curvature bound to control the curvature terms in the Bochner formula of these Dirac operators.

\par To prove Theorem \ref{mm2}, we consider the twisted Dirac operators $D_i \otimes B_{k}(T_\C M)$, which appears in the $q$-expansion of the Witten operators of elliptic genera, and their Witten deformations.
 Here $D_i$ is the Atiyah-Singer spin Dirac operator on $M$
 and $B_{k}(T_\C M)$ is an integral linear combination of bundles
 of type $S^{i_1}(T_\CC M)\otimes \cdots \otimes S^{i_r}(T_\CC M)\otimes \Lambda^{j_1}(T_\CC M)\otimes \cdots \otimes\Lambda^{j_s}(T_\CC M)$,
  which are subbundles of tensor products of $T_\CC M$ of power at most $k$. See Section 5 for more details.
The rigidity of these operators is a celebrated conjecture of Witten which was proved in \cite{BT, Liu95, Liu96, Tau}. The rest part of proof is almost identical as Theorem \ref{mm0}, except that in Step 2
we need  an additional integral curvature bound to control the curvature terms in the Bochner formula of these twisted Dirac operators.
\par
A special case of Theorems \ref{mm0}-\ref{mm2} is: let $M$ be a compact manifold satisfying the
topological assumptions in Theorem \ref{mm0}-\ref{mm2}, then the isometry group of any Riemannian metric
$g$ on $M$ with nonpositive Ricci curvature is finite (in Theorem \ref{mm0} we also assume that $g$ is
  K\"ahler). This is actually easy to prove. Otherwise if $(M, g)$ has infinite isometry group, then
 $M$ admits a nowhere vanishing Killing vector field by Bochner's theorem, which implies that all topological invariants in Theorems \ref{mm0}-\ref{mm2} must vanish.
  However, under the much weaker curvature assumptions in
Theorems \ref{mm0}-\ref{mm2}, one will only get a generally nonzero Killing vector field on $M$ which might have zeros.
 It is crucial to use the rigidity of those Dirac operators and a mean value inequality to get around the difficulty.

\section*{Acknowledgements} Xiaoyang Chen is partially supported by National Natural Science Foundation of China No.12171364.
 He thanks Prof. Binglong Chen and Prof. Botong Wang for helpful discussions.  Fei Han is partially supported by the grant AcRF R-146-000-263-114 from National University of Singapore.
He is indebted to Prof. Kefeng Liu and Prof. Weiping Zhang for helpful discussions. Both authors would like to thank the Mathematical Science Research Center at Chongqing University of Technology for hospitality during their visit  and thank Prof. Wilderich Tuschmann for the helpful discussion.

\section{Dirac bundles and an integral formula}
In this section, we briefly review Dirac bundles (page 114 in \cite{LM}) and then prove an integral formula as well as an inequality, which will finish the
 first step in the proof of Theorems \ref{mm0}-\ref{mm2}.
\par
Let $M$ be a compact Riemannian manifold of dimension $m$ and $\nabla^{TM}$ be the Levi-Civita connection. Let $Cl(M)$ be Clifford algebra bundle constructed from the the
tangent bundle $TM$ and the Riemannian metric. $\nabla^{TM}$ induces a connection on $Cl(M)$, which we will still denote by $\nabla^{TM}$. Let $E$ be a complex vector bundle
of left module over $Cl(M)$ (i..e. a vector bundle over $M$ such that at each point $x\in M$, the fiber $E_x$ is a left module over the algebra $Cl(M)_x$). $E$ together with a
Hermitian metric $g^E$ and a  compatible connection $\nabla^E$ is called a  Dirac bundle if \newline
(i) The Clifford multiplication by unit tangent vectors is unitary, i.e., for each $x\in M$,
\be \langle c(e)s_1, c(e)s_2\rangle=\langle s_1, s_2\rangle \ee
for all $s_1, s_2\in E_x$ and  unit vectors $e\in T_x M$; this is equivalent to
\be  \langle c(e)s_1, s_2\rangle+\langle s_1, c(e)s_2\rangle=0\ee
for all  $s_1, s_2\in E_x$ and  unit vectors $e\in T_x M$; \newline
(ii)  The connection $\nabla^E$ is a module derivation, i.e.,
\be \nabla^E(\phi\cdot s)=(\nabla^{TM}\phi)\cdot s+\phi\cdot (\nabla^E s)\ee
for all $\phi\in \Gamma(Cl(M))$ and all $s\in \Gamma(E)$.

The Dirac operator on $E$ is the first-order differential operator $D: \Gamma(E)\to \Gamma(E)$ defined by setting
\be D s=\sum_{j=1}^m c(e_j)\nabla^E_{e_j}s \ee
 where $e_1, e_2, \cdots, e_m$ is a local orthonormal  basis of $T M$.
On $\Gamma(E)$, there is an inner product induced from the pointwise inner product by setting
$$(s_1, s_2)=\int_M \langle s_1, s_2\rangle. $$ The Dirac operator is formally self-adjoint with respect to this inner product, i.e.,
\be (D s_1, s_2)=(s_1, D s_2)\ee
for any sections $s_1, s_2$.

\par
Let $X$ be a tangent vector field on $M$.  Suppose $s \in \Gamma(E)$ satisfies
$$(D+\sqrt{-1}t c(X))s=0$$
for some $ t \in \mathbb{R}$.

Then we have the following integral formula.
\begin{thm} \label{int0}
\be 2\sqrt{-1}\int_M t |X|^2 |s|^2= \int_M - 2\llag\nabla^{E}_X s, s \rrag-\sum_{i=1}^m \langle c(\nabla^{TM}_{e_i}X)s, c(e_i)s \rangle. \ee
\end{thm}
\begin{proof} Let  $\{e_i\}$ be a local orthonormal basis. Define a vector field $U$ by
$$U=\sum_{i=1}^m \langle c(X)s, c(e_i)s \rangle e_i.$$
Then
\be
\begin{split}
&\divg U =\sum_{j=1}^m \left\langle\nabla^{TM}_{e_j}\left( \sum_{i=1}^m \langle c(X)s, c(e_i)s\rangle e_i\right), e_j\right\rangle\\
=&\sum_{i=1}^m \left\langle \nabla^{E}_{e_i}(c(X)s), c(e_i)s \right\rangle+\sum_{i=1}^m \llag c(X)s,\nabla^{E}_{e_i}(c(e_i)s)\rrag +
\sum_{i,j=1}^m \llag c(X)s,c(e_i)s \rrag \llag \nabla^{TM}_{e_j}e_i, e_j \rrag \\
=&\sum_{i=1}^m \left\langle \nabla^{E}_{e_i}(c(X)s), c(e_i)s \right\rangle+\sum_{i=1}^m \llag c(X)s, c(e_i)\nabla^{E}_{e_i}s\rrag
\\
&+\sum_{i=1}^m \llag c(X)s, c(\nabla^{TM}_{e_i}e_i) s\rrag
 -\sum_{i,j=1}^m \llag c(X)s,c(e_i)s \rrag \llag \nabla^{TM}_{e_j}e_j, e_i \rrag \\
=&\sum_{i=1}^m \llag c(\nabla^{TM}_{e_i}X)s +c(X)\nabla^{E}_{e_i}s, c(e_i)s\rrag+\sum_{i=1}^m \llag c(X)s, c(e_i)\nabla^{E}_{e_i}s \rrag\\
=&\sum_{i=1}^m \llag c(\nabla_{e_i}^{TM}X)s, c(e_i)s \rrag+\sum_{i=1}^m \llag c(X)\nabla^{E}_{e_i}s, c(e_i)s \rrag+\llag c(X)s, Ds\rrag\\
=&\sum_{i=1}^m \llag c(\nabla_{e_i}^{TM}X)s, c(e_i)s\rrag-\sum_{i=1}^m \llag c(e_i)c(X)\nabla^{E}_{e_i}s, s\rrag+\llag c(X)s, Ds \rrag\\
=&\sum_{i=1}^m \llag c(\nabla_{e_i}^{TM}X)s, c(e_i)s \rrag+\sum_{i=1}^m \llag (c(X)c(e_i)+2\llag e_i, X\rrag)\nabla^{E}_{e_i}s, s\rrag+\llag c(X)s, Ds \rrag\\
=&\sum_{i=1}^m \llag c(\nabla_{e_i}^{TM}X)s, c(e_i)s \rrag+\llag c(X)D s, s \rrag+2\llag\nabla^{E}_X s, s \rrag+\llag c(X)s, Ds\rrag.\\
\end{split}
\ee

But since $ Ds=-\sqrt{-1}t c(X)s$, we have
\be  \llag c(X)s, Ds\rrag=\sqrt{-1}t |c(X)s|^2=\sqrt{-1}t|X|^2 |s|^2;\ee
\be \llag c(X)Ds, s\rrag=-\sqrt{-1}\llag t c(X)c(X)s, s\rrag=\sqrt{-1}t|X|^2 |s|^2. \ee

The desired formula follows.
\end{proof}

Now we apply the integral formula in  Theorem \ref{int0} to prove inequality \ref{1111}, which will finish the first step of proof
of Theorems \ref{mm0}-\ref{mm2}.
Let $g_i$ be a sequence of Riemannian metrics on $M$ with infinite isometry groups.
Now consider the following Dirac bundles and operators discussed in the introduction.

\par (1) In Theorem \ref{mm0}, we consider the Dirac bundles $E=\oplus_{p} \wedge^{2p,0}(M)$ and Dirac operators
$P_i=\bar{\partial_i}+\bar{\partial_i}^*$.
\par (2) In Theorem \ref{mm1}, we  consider the Dirac bundles $E=\oplus_{p} \wedge^{2p}(M) \otimes \mathbb{C}$ and Dirac operators
$P_i=d+d^*$ or $E$ is the space of self dual differential forms and $P_i$ is the signature operator.
\par (3) In Theorem \ref{mm2}, we  consider the Dirac bundles $E=S(TM) \otimes B_{k}(T_\C M))$ for some $k \leq {[\frac{n}{2}]}$ and Dirac operators
$P_i=D_i \otimes B_{k}(T_\C M)$. Here $D_i$ are the Atiyah-Singer spin Dirac operators on $M$.
 $S(TM)$ is the spinor bundle over a compact $4n$ dimensional spin manifold $M$
and $B_{k}(T_\C M)$ involve linear combinations of tensor product of $T_\C M$ at most to power $k$.
See section 5 for more information.

\par
Let  $X_i$ be a nonzero Killing vector field generating an isometric $S^1$ action on $M$.
Suppose $s_i \in \Gamma(E)$ satisfies
$$(P_i+\sqrt{-1}t_i c(X_i))s=0$$
$$L_{X_i} s_i=0,$$
where $ t_i \in \mathbb{R}$ and $L_{X_i} s_i$ is the Lie derivative of $s_i$ in the direction $X_i$.
Then we have the following crucial inequality.

\begin{thm} \label{int}
\begin{equation} \label{1}
\int_{M}  t_i^2 |X_i|^2 |s_i|^2 dV_i \leq C(n) \int_{M} t_i |\nabla X_i| |s_i|^2 dV_i,
\end{equation}
where $C(n)$ is some constant depending only on $n$.
\end{thm}

\begin{proof}
Theorem \ref{int} is a direct consequence of Theorem \ref{int0} and the following inequality:
\begin{equation} \label{inn}
\llag\nabla_{X_i} s_i, s_i \rrag -\llag L_{X_i} s_i, s_i \rrag \leq C(n) |\nabla X_i| |s_i|^2,
\end{equation}
where  $C(n)$ is some positive constant depending only on $n$.
\\
\par

(1)  When $P_i=\bar{\partial_i}+\bar{\partial_i}^*$ or $P_i=d+d^*$ or the signature operator, inequality \ref{inn} is an easy consequence of
the torsion free property of the Levi-Civita connection $\nabla^{TM}$.

\par

(2) For the Dirac operators $P_i=D_i \otimes B_{k}(T_\C M)$, by (1.24) in  \cite{TZ}, we get
  $$L_{X_i}|_{S(TM)}-\nabla_{X_i} ^{S(TM)}=-\sum_{j,k=1}^{4n}\frac{1}{4}\llag \nabla_{e_j}^{TM} X_i, e_k \rrag c(e_j)c(e_k).$$
 As $\nabla^{TM}$ is torsion free, we have
  $$ L_{X_i} -\nabla^{TM}_{X_i} =-\nabla^{TM} X_i.$$
 Since by Theorem \ref{mod 2}, $B_{k}(T_{\mathbb{C}} M)$ is an integral linear combination of bundles of type
$$S^{i_1}(T_\CC M)\otimes \cdots \otimes S^{i_r}(T_\CC M)\otimes \Lambda^{j_1}(T_\CC M)\otimes \cdots \otimes \Lambda^{j_s}(T_\CC M),$$ which are subbundles of
tensor products of $T_\CC M$ of power at most $k, 0 \leq k \leq [\frac{n}{2}]$,
we see that
$$\llag\nabla_{X_i} s_i, s_i \rrag -\llag L_{X_i} s_i, s_i \rrag \leq C(n) |\nabla X_i| |s_i|^2$$
for some constant $C(n)$ depending only on $n$.

\end{proof}

\section{A mean value inequality}
In this section we prove a mean value inequality which will be used in the proof of Theorem \ref{mm0}-\ref{mm2}.
let $g_i$ be a sequence of Riemannian metrics on $M$ with infinite isometry groups
and $X_i$ a nonzero Killing vector field on $M$.
Consider the Dirac operators $P_i$ discussed in section 2 and their Witten deformations:
\be \label{deform} \widetilde{P_i}=P_i + \sqrt{-1} t_i c(X_i),\ee
where $t_i:=(\frac{V(g_i)}{\int_{M_i} |X_i|^2 dV_i})^{1/2}>0$.
\par
When $P_i=\bar{\partial_i}+\bar{\partial_i}^*$, we assume that $M$ has nonzero holomorphic Euler number and $g_i$ is also K\"ahler and satisfy
$$-\lambda_1 \leq Ric (g_i) \leq \frac{1}{i}, \ diam (g_i) \leq 1.$$
In other cases, we assume that $M$ has nonzero Euler number / signature / elliptic genera and $g_i$ satisfy
  $$-\lambda_1 \leq Ric (g_i) \leq \frac{1}{i}, \ diam (g_i) \leq 1, \frac{1}{V(g_i)}\int_M |\Re_{g_i}|^p dV_i
\leq \lambda_2.$$

 Then there is $s_i \in \Gamma(E)$ satisfying
$$s_i \neq 0$$
$$(P_i+\sqrt{-1}t_i c(X_i))s_i=0$$
$$L_{X_i} s_i=0.$$

Moreover, we have
\begin{thm} \label{good2}
\begin{equation} \label{211}
|X_i|^2_{\infty}=:max_{x \in M}|X_i|^2(x) \leq C_1 \frac{\int_{M} |X_i|^2 dV_i}{V(g_i)},
\end{equation}
\begin{equation} \label{212}
|s_i|^2_{\infty}=:max_{x \in M}|s_i|^2(x) \leq C_2 \frac{\int_{M} |s_i|^2 dV_i}{V(g_i)},
\end{equation}
where $C_1, C_2$ are two constants depending on $n, p,  \lambda_1, \lambda_2$.
\end{thm}

\begin{proof}
Theorem \ref{good2} is in fact a consequence of a general mean value inequality in Theorem \ref{sbl} below.
Since $Ric (g_i) \leq \frac{1}{i}$, applying Bochner formula to $X_i$, we get
\begin{equation} \label{413}
\frac{1}{2}\Delta |X_i|^2=|\nabla X_i|^2 - Ric(g_i)(X_i,X_i)\geq |\nabla X_i|^2 - \frac{1}{i}|X_i|^2,
\end{equation}
where $\Delta$ is the Laplacian acting on functions which is a negative operator.
On the other hand, by Kato's inequality \cite{Br}, we have $|\nabla X_i|\geq |\nabla|X_i||$. It follows that
\begin{equation} \label{414}
|X_i| \Delta |X_i| \geq  -\frac{1}{i}|X_i|^2.
\end{equation}
Since $Ric(g_i) \geq  -\lambda_1, diam (g_i) \leq 1$, applying Theorem \ref{sbl} to $|X_i|$, we get
\begin{equation} \label{17}
|X_i|^2_{\infty}=:max_{x \in M}|X_i|^2(x) \leq C_1 \frac{\int_{M} |X_i|^2 dV_i}{V(g_i)},
\end{equation}
where $C_1$ is some constant depending on $n, \lambda_1$.

\par
To prove the mean value inequality of $s_i$, applying the Bochner formula to $s_i$, we get
\begin{equation} \label{13}
\frac{1}{2}\Delta |s_i|^2 = |\nabla s_i|^2  -  \langle P_i^2 s_i, s_i \rangle  + \langle \Psi_i s_i, s_i \rangle,
\end{equation}
where $\Psi_i $ is a symmetric endomorphism of the Dirac bundles $E$ discussed in section 2.
\par Define a vector field $Y_i$ by the condition
$$\langle Y_i, W \rangle =-\langle P_i s_i, c(W)s_i \rangle.$$
Then by the proof of Proposition 5.3 in pages 114-115, \cite{LM}, we get
$$\langle P_i^2 s_i, s_i \rangle = \langle P_i s_i, P_i s_i \rangle + div Y_i.$$
As $P_i s_i + \sqrt{-1} t_i c(X_i) s_i=0$, then we have
$$\frac{1}{2}\Delta |s_i|^2 = |\nabla s_i|^2 - \langle P_i s_i, P_i s_i \rangle - div Y_i + \langle \Psi_i s_i, s_i \rangle$$
$$=|\nabla s_i|^2 -  |t_i X_i|^2 |s_i|^2  - div Y_i + \langle \Psi_i s_i, s_i \rangle.$$

\par (1) When $P_i=\bar{\partial_i}+\bar{\partial_i}^*$, the curvature term $\langle \Psi_i s_i, s_i \rangle$ can be controlled
by the Ricci curvature, see \cite{M}. Since $Ric(g_i) \geq  -\lambda_1$, we get
$$\langle \Psi_i s_i, s_i \rangle \geq -C(n) \lambda_1 |s_i|^2,$$
where $C(n)$ is some constant  depending only $n$.
For any $x \in M$, by the choice of $t_i$, we have
$$|t_i X_i|^2 (x) \leq t_i^2 |X_i|^2_{\infty} \leq t_i^2 C_1 \frac{\int_{M} |X_i|^2 dV_i}{V(g_i)}=C_1.$$
Hence we get

\begin{equation} \label{18}
\frac{1}{2}\Delta |s_i|^2 \geq |\nabla s_i|^2  - (C_1  + C(n) \lambda_1) |s_i|^2-div Y_i
\end{equation}
By Kato's inequality, we have $|\nabla s_i|\geq |\nabla |s_i||$. It follows that
\begin{equation} \label{414}
|s_i| \Delta |s_i| \geq  - (C_1 +  C(n)\lambda_1) |s_i|^2-div Y_i
\end{equation}
By the definition of $Y_i$, we get
$$|Y_i| \leq t_i |X_i| |s_i|^2 \leq C_1^{\frac{1}{2}} |s_i|^2.$$
Since $Ric(g_i) \geq  -\lambda_1, diam (g_i) \leq 1$, applying Theorem \ref{sbl} to $|s_i|$, we get
$$|s_i|^2_{\infty}=:max_{x \in M}|s_i|^2(x) \leq  C_2 \frac{\int_{M} |s_i|^2 dV_i}{V(g_i)}$$
for some positive constant $C_2$ depending only on  $n,  \lambda_1$.

\par (2)  When $P_i=d+d^*$ or the signature operator or $D_i \otimes B_{k}(T_\C M)$, we need to control the curvature term $\langle \Psi_i s_i, s_i \rangle$
by the full Riemannian curvature tensor $|\Re_{g_i}|$ \cite{LM}. More precisely,  we have
$$\langle \Psi_i s_i, s_i \rangle \geq - C(n) |\Re_{g_i}| |s_i|^2$$
where $C(n)$ is some constant  depending only $n$.

\par

For any $x \in M$, by the choice of $t_i$, we have
$$|t_i X_i|^2 (x) \leq t_i^2 |X_i|^2_{\infty} \leq t_i^2 C_1 \frac{\int_{M} |X_i|^2 dV_i}{V(g_i)}=C_1.$$
Hence we get

\begin{equation} \label{18}
\frac{1}{2}\Delta |s_i|^2 \geq |\nabla s_i|^2  - (C_1 + C(n) |\Re_{g_i}|) |s_i|^2-div Y_i
\end{equation}
By Kato's inequality, we have $|\nabla s_i|\geq |\nabla |s_i||$. It follows that
\begin{equation} \label{414}
|s_i| \Delta |s_i| \geq  - (C_1 +  C(n) |\Re_{g_i}|) |s_i|^2-div Y_i
\end{equation}
By the definition of $Y_i$, we get
$$|Y_i| \leq t_i |X_i| |s_i|^2 \leq C_1^{\frac{1}{2}} |s_i|^2.$$
Since $Ric(g_i) \geq  -\lambda_1, diam (g_i) \leq 1, \frac{1}{V(g_i)}\int_M |\Re_{g_i}|^p dV_i
\leq \lambda_2$, applying Theorem \ref{sbl} to $|s_i|$, we get
$$|s_i|^2_{\infty}=:max_{x \in M}|s_i|^2(x) \leq   C_2 \frac{\int_{M} |s_i|^2 dV_i}{V(g_i)}$$
for some constant $C_2$ depending only on  $n, p,  \lambda_1, \lambda_2$.

\end{proof}

Now we prove a general mean value inequality.
We firstly recall the following  Poincar\'e-Sobolev inequality, see for example Theorem 2,  page 386 and Theorem 3, page 397 in  \cite{Br}.
\begin{thm} \label{sob}
Let $(M,g)$ be a closed $m$-dimensional smooth Riemannian manifold such that for some constant $b>0$,
$$ r_{min}(g) (diam(g))^2 \geq-(m-1)b^2,$$
where $diam (g)$ is the diameter of $g$, $Ric(g)$ is the Ricci curvature of $g$ and
$$r_{min}(g)=inf\{{Ric(g)(u,u): u \in TM, g(u,u)=1}\}.$$
Let $R=\frac{diam(g)}{b C(b)}$, where $C(b)$ is the unique positive root of the equation
$$ x \int_0^{b}(cht + x sht)^{m-1} dt=\int_0^{\pi} sin^{m-1}t dt.$$
Then for each $1 \leq l_1 \leq \frac{ml_2}{m-l_2}, l_1 < \infty$ and $f \in W^{1,l_2}(M)$, we have
$$ \|f-\frac{1}{V(g)}\int_M f dV \|_{l_1} \leq S_{l_1,l_2} \|\nabla f\|_{l_2} $$
$$ \|f\|_{l_1} \leq S_{l_1, l_2} \|\nabla f\|_{l_2} + V(g)^{1/l_1-1/l_2}\|f\|_{l_2},$$
where $V(g)$ is the volume of $(M,g)$, $S_{l_1,l_2}=(V(g)/vol(S^m (1))^{1/l_1-1/l_2} R \Sigma(m, l_1, l_2)$ and $\Sigma(m, l_1, l_2)$ is the Sobolev constant of
the canonical unit sphere $S^m$ defined by
$$\Sigma(m, l_1, l_2)=sup\{{\|f\|_{l_1}/\|\nabla f\|_{l_2}: f \in W^{1,l_2}(S^m), f \neq 0, \int_{S^m}f=0}\}.$$
\end{thm}

\par
As an application of Theorem \ref{sob}, we get the following mean value inequality which is a generalization of Theorem  3 in \cite{Br}, pages 395-396.
See also \cite{PL} pages 80-84.

\begin{thm}\label{sbl}
Let $(M,g)$ be a closed $m$-dimensional smooth Riemannian manifold such that for some constant $b>0$,
$$r_{min}(g) (diam(g))^2 \geq-(m-1)b^2.$$
If $f\in W^{1,2}(M)$ is a nonnegative continuous functions such that $f \Delta f \geq -h_1 f^2-div Y$ (here $\Delta$ is a negative operator)
in the sense of distribution for some nonnegative continuous function $h_1$ and $Y$ is a $C^1$ vector field satisfying
$$|Y|(x) \leq h_2(x) f^2(x), \forall x\in M$$
for some nonnegative continuous function $h_2$, then
$$max_{x \in M}|f|^2(x) \leq C(m, p, R, \Lambda) \frac{\int_M f^2 dV}{V(g)},$$
where $C(m, p, R, \Lambda)$ is some constant depending only on $m, p, R=\frac{diam(g)}{b C(b)}$ and
$$\Lambda=\frac{\int_M h^p dV}{V(g)}, p > \frac{m}{2}$$
$$h=h_1 +2 h_2^2.$$
 \end{thm}

\begin{proof}
The proof is a standard application of Moser iteration.
For any $k \geq 1$, multiply the inequality $f \Delta f \geq -h_1 f^2 -div Y$ by $f^{2k-2}$ and integrate. Then we get
$$\int_M f^{2k-1} \Delta f \geq \int_M -h_1 f^{2k} - div Y \ f^{2k-2}$$
$$=\int_M -h_1 f^{2k} + \langle Y, \nabla f^{2k-2}\rangle$$
$$= \int_M -h_1 f^{2k} + (2k-2) f^{2k-3} \langle Y, \nabla f \rangle$$
Hence
$$(2k-1) \int_M f^{2k-2} |\nabla f|^2 \leq \int_M h_1 f^{2k} - (2k-2) f^{2k-3} \langle Y, \nabla f \rangle$$
$$\leq \int_M h_1 f^{2k} + (2k-2) f^{2k-1} h_2 |\nabla f| $$
$$\leq  \int_M h_1 f^{2k} + (2k-2)  h_2^2 f^{2k} + \frac{2k-2}{4} f^{2k-2} |\nabla f|^2.$$
Then
$$\frac{3k-1}{2} \int_M f^{2k-2} |\nabla f|^2 \leq \int_M (h_1 + (2k-2)  h_2^2) f^{2k}$$
And
$$\int_M |\nabla f^k|^2 \leq \frac{2k^2}{3k-1}\int_M (h_1 + (2k-2)  h_2^2) f^{2k}$$
$$\leq k^2 \int_M (h_1 + 2 h_2^2) f^{2k}.$$
So
$$ \|\nabla f^k\|_2 \leq  (\int_M k^2 h  f^{2k})^{\frac{1}{2}}.$$

Let $v=\frac{m}{2}$ if $m>2$ and $ 1 < v < p $ be arbitrary for  $m=2$. Let $\mu$ be the conjugate of $v$ such that
$$\frac{1}{v} + \frac{1}{\mu}=1.$$
Applying Theorem \ref{sob} to $f^k$, we get

\begin{equation} \label{new0}
\|f^k\|_{2 \mu} \leq S_{2\mu,2} \|\nabla f^k \|_2 + V(g)^{\frac{1-\mu}{2 \mu}}\|f^k \|_2.
\end{equation}

Let $A=(\int_M h^p)^{\frac{1}{p}}$. Since $p> \frac{m}{2}$, we see that $p >v$ by the choice of $v$. By the H$\ddot{o}$lder inequality, we have
$$k^2 \int_M h f^{2k} \leq k^2 A (\int_M (f^{2k})^{\frac{p}{p-1}})^{\frac{p-1}{p}}$$
\begin{equation} \label{new}
\leq k^2 A (\int_M f^{2k})^{\frac{\mu(p-1)-p}{p(\mu-1)}}  (\int_M  f^{2k \mu})^{\frac{1}{p(\mu-1)}},
\end{equation}

Define $\epsilon, \delta, y$ by
$$\epsilon=\frac{\mu (p-1)-p}{p(\mu-1)}$$
$$(\delta {\epsilon}^{\frac{1}{1-\epsilon}} (\frac{1}{\epsilon}-1))^{\frac{1}{2}}=\frac{1}{2  S_{2\mu, 2}}$$
$$y=(k^2 A)^{\frac{p (\mu-1)}{\mu(p-1)-p}} ( \int_M f^{2k}) (\int_M f^{2k\mu})^{-\frac{1}{\mu}}.$$
Then $0<\epsilon <1$ as $p > v$.
By Young inequality, we get
$$ y^{\epsilon} \leq \delta^{\frac{\epsilon -1}{\epsilon}} y + \delta {\epsilon}^{\frac{1}{1-\epsilon}} (\frac{1}{\epsilon}-1).$$
Hence
$$k^2 A (\int_M f^{2k} )^{\frac{\mu (p-1)-p}{p(\mu-1)}} (\int_M f^{2k \mu})^{\frac{p-\mu (p-1)}{p \mu (\mu-1)}}$$
$$\leq \delta^{\frac{\epsilon-1}{\epsilon}} (k^2 A)^{\frac{p (\mu-1)}{\mu(p-1)-p}} (\int_M f^{2k}) (\int_M f^{2k \mu})^{-\frac{1}{\mu}}+ \delta
{\epsilon}^{\frac{1}{1-\epsilon}} (\frac{1}{\epsilon}-1).$$
Multiplying through by $ (\int_M f^{2k\mu})^{\frac{1}{\mu}}$, combined with (\ref{new}),
we get
$$k^2 \int_M h f^{2k} \leq \delta^{\frac{\epsilon-1}{\epsilon}} (k^2 A)^{\frac{p (\mu-1)}{\mu(p-1)-p}} \int_M f^{2k} + \delta {\epsilon}^{\frac{1}{1-\epsilon}}
(\frac{1}{\epsilon}-1) (\int_M f^{2k\mu})^{\frac{1}{\mu}}. $$
Then
\begin{equation} \label{new2}
(k^2 \int_M h f^{2k})^{\frac{1}{2}} \leq \delta^{\frac{\epsilon-1}{2 \epsilon}} (k^2 A)^{\frac{1}{2} \frac{p (\mu-1)}{\mu(p-1)-p}} (\int_M f^{2k})^{\frac{1}{2}} +
(\delta {\epsilon}^{\frac{1}{1-\epsilon}} (\frac{1}{\epsilon}-1))^{\frac{1}{2}} (\int_M f^{2k\mu})^{\frac{1}{2\mu}}.
\end{equation}

Combined with (\ref{new0}), we get
$$(\int_M f^{2k \mu})^{\frac{1}{2\mu}} \leq S_{2\mu, 2} (k^2 \int_M h f^{2k})^{\frac{1}{2}} +  V(g)^{\frac{1-\mu}{2 \mu}} (\int_M f^{2k})^{\frac{1}{2}}$$
$$\leq S_{2\mu, 2} \ \delta^{\frac{\epsilon-1}{2 \epsilon}} (k^2 A)^{\frac{1}{2} \frac{p (\mu-1)}{\mu(p-1)-p}} (\int_M f^{2k})^{\frac{1}{2}} +
S_{2\mu, 2} \ (\delta {\epsilon}^{\frac{1}{1-\epsilon}} (\frac{1}{\epsilon}-1))^{\frac{1}{2}} (\int_M f^{2k\mu})^{\frac{1}{2\mu}}+ V(g)^{\frac{1-\mu}{2 \mu}} (\int_M
f^{2k})^{\frac{1}{2}}.$$

As $(\delta {\epsilon}^{\frac{1}{1-\epsilon}} (\frac{1}{\epsilon}-1))^{\frac{1}{2}}=\frac{1}{2  S_{2\mu, 2}}$, then
$\delta=C(m,p) (\frac{1}{S_{2\mu, 2}})^2$ for some constant $C(m, p)$ depending only on $m, p$. Moreover, we have
$$(\int_M f^{2k \mu})^{\frac{1}{2\mu}}
\leq  2 S_{2\mu, 2} \ \delta^{\frac{\epsilon-1}{2 \epsilon}} (k^2 A)^{\frac{1}{2} \frac{p (\mu-1)}{\mu(p-1)-p}} (\int_M f^{2k})^{\frac{1}{2}}
+ 2 V(g)^{\frac{1-\mu}{2 \mu}} (\int_M f^{2k})^{\frac{1}{2}}.$$
Then
$$\|f\|_{{2k\mu}}\leq (2 S_{2\mu, 2} \ \delta^{\frac{\epsilon-1}{2 \epsilon}} (k^2 A)^{\frac{1}{2} \frac{p (\mu-1)}{\mu(p-1)-p}}+ 2
V(g)^{\frac{1-\mu}{2 \mu}})^{\frac{1}{k}}\|f\|_{2k}.$$
By the choice of $\epsilon$, we have
$$\frac{\epsilon-1}{2 \epsilon}=\frac{-\mu}{2(\mu(p-1)-p)}.$$
As $S_{2\mu, 2}=C(m, p) V(g)^{\frac{1-\mu}{2\mu}} R$ for some constant $C(m, p)$ depending only on $m, p$, then
$$\|f\|_{{2k\mu}}\leq (C(m,p) (V(g)^{\frac{1-\mu}{2 \mu}} R)^{\frac{p(\mu-1)}{\mu (p-1)-p}} (k^2 A)^{\frac{1}{2} \frac{p (\mu-1)}{\mu(p-1)-p}} + 2
V(g)^{\frac{1-\mu}{2 \mu}})^{\frac{1}{k}}\|f\|_{2k}$$
\begin{equation} \label{new6}
\leq B^{\frac{1}{k}} k^{\frac{1}{k} \frac{p (\mu-1)}{\mu(p-1)-p}} V(g)^{\frac{1-\mu}{2 \mu k }}\|f\|_{2k},
\end{equation}
where $$B=C(m,p) V(g)^{\frac{\mu-1}{2\mu} \frac{-\mu}{\mu(p-1)-p}} R^{\frac{p(\mu-1)}{\mu (p-1)-p}} A^{\frac{1}{2} \frac{p (\mu-1)}{\mu(p-1)-p}} + 2
=C(m,p) \Lambda^{\frac{1}{2} \frac{\mu-1}{\mu(p-1)-p}} R^{\frac{p(\mu-1)}{\mu (p-1)-p}} +2.$$
Let $k={\mu}^i, i=0,1,\cdots$. Since $K_1= \sum i {\mu}^{-i}$ and $K_2=\sum {\mu}^{-i}$  is finite, multiplying (\ref{new6}), we get
$$max_{x \in M}|f|^2(x) \leq C(m, p, R, \Lambda) \frac{\int_M f^2 dV}{V(g)},$$
$$C(m, p, R, \Lambda)=\mu^{2K_1 \frac{p(\mu-1)}{\mu (p-1)-p}} B^{2 K_2}.$$

\end{proof}

\section{Proof of Theorems \ref{mm0}-\ref{mm2}}

Now we finish the proof of Theorems \ref{mm0}-\ref{mm2} by contradiction simultaneously.
To prove Theorem \ref{mm0},
let $g_i$ be a sequence of K\"ahler  metrics on a
 compact complex manifold $M$ with nonzero holomorphic Euler number such that
 the isometry groups of $(M, g_i)$ are infinite and
 $$-\lambda_1 \leq Ric (g_i) \leq \frac{1}{i}, \ diam (g_i) \leq 1.$$
 To prove Theorem \ref{mm1} or \ref{mm2}, then we assume that $M$ has nonzero Euler number /  signature / elliptic genera and
 $g_i$ is a sequence of Riemannian metrics on $M$ with infinite isometry groups and satisfy
  $$-\lambda_1 \leq Ric (g_i) \leq \frac{1}{i}, \ diam (g_i) \leq 1, \frac{1}{V(g_i)}\int_M |\Re_{g_i}|^p dV_i
\leq \lambda_2.$$
Let $X_i$ be a nonzero Killing vector field on $M_i$.
Consider the Dirac operators $P_i$ as discussed in section 2 and their Witten deformations:
\be \label{deform} \widetilde{P_i}=P_i + \sqrt{-1} t_i c(X_i),\ee
where $t_i:=(\frac{V(g_i)}{\int_{M} |X_i|^2 dV_i})^{1/2}>0$.
\par Then there is $s_i \in \Gamma(E)$ satisfying
$$s_i \neq 0$$
$$(P_i+\sqrt{-1}t_i c(X_i))s_i=0$$
$$L_{X_i} s_i=0.$$

\begin{lem} \label{co}
 For sufficiently large $i$, we have
$$ \int_{M} |s_i|^2 dV_i \leq  \frac{1}{2}\int_{M} |s_i|^2 dV_i.$$
\end{lem}

By Lemma \ref{co}, $s_i \equiv 0$ for sufficiently large $i$, which is a contradiction.

\par The proof of Lemma \ref{co} will be based on the following lemmas.
\begin{lem} \label{boc}
\begin{equation}
\int_{M} |\nabla X_i|^2 dV_i \leq \frac{1}{i} \int_{M} |X_i|^2 dV_i.
\end{equation}
\end{lem}

\begin{proof}
As $Ric(g_i) \leq  \frac{1}{i}$, applying Bochner formula to $X_i$ \cite{P}, we get
\begin{equation} \label{4.2}
\frac{1}{2}\Delta |X_i|^2=|\nabla X_i|^2 - Ric(g_i)(X_i,X_i)\geq |\nabla X_i|^2 - \frac{1}{i}|X_i|^2,
\end{equation}
where $\Delta$ is the Laplacian acting on functions which is a negative operator.
Then
\begin{equation} \label{20}
\int_{M} |\nabla X_i|^2 dV_i \leq \frac{1}{i} \int_{M} |X_i|^2 dV_i.
\end{equation}
\end{proof}

\begin{lem} \label{good1}
$$\int_{M} t_i^2 |X_i|^2 |s_i|^2 dV_i
\leq \frac{C(n)}{\sqrt{i}} |s_i|_{\infty} (\int_{M} t_i^2 |X_i|^2 dV_i )^{\frac{1}{2}} (\int_{M}|s_i|^2 dV_i)^{\frac{1}{2}}$$
where $|s_i|_{\infty}=\max_{x \in M}|s_i|(x)$ and
$C(n)$ is some constant depending only on $n$.
\end{lem}

\begin{proof}

By Theorem \ref{int} and Lemma \ref{boc}, we get
$$\int_{M} t_i^2 |X_i|^2 |s_i|^2 dV_i \leq C(n) \int_{M}  t_i|\nabla X_i||s_i|^2 dV_i $$
$$\leq C(n) (\int_{M} t_i^2 |\nabla X_i|^2 dV_i )^{\frac{1}{2}} (\int_{M}|s_i|^4 dV_i)^{\frac{1}{2}} $$
\begin{equation} \label{21}
\leq \frac{C(n)}{\sqrt{i}} |s_i|_{\infty} (\int_{M} t_i^2 |X_i|^2 dV_i )^{\frac{1}{2}} (\int_{M}|s_i|^2 dV_i)^{\frac{1}{2}},
\end{equation}
where $|s_i|_{\infty}=\max_{x \in M}|s_i|(x).$

\end{proof}

\begin{lem} \label{good3}

\begin{equation}
\frac{ \int_{M}  |X_i|^2 dV_i}{V(g_i)} \int_{M} |s_i|^2 dV_i \leq  \int_{M}  |X_i|^2 |s_i|^2 dV_i  +
\frac{ C_1 |s_i|^2_{\infty}   }{\sqrt{i}}  \int_{M} |X_i|^2 dV_i
\end{equation}
for some constant $C_1$ depending on $n, p, \lambda_1, \lambda_2$.
\end{lem}
\begin{proof}
Let $h_i=|X_i|^2$ and $\overline{h_i}= \frac{\int_{M}  |X_i|^2 dV_i }{V(g_i)}$. Since $diam (g_i) \leq 1,  Ric(g_i) \geq -\lambda_1$,
by Theorem \ref{good2}, \ref{sob} and Lemma \ref{boc}, we get
$$\int_{M} |h_i - \overline{h_i}||s_i|^2 dV_i \leq |s_i|^2_{\infty} (\int_{M} |h_i - \overline{h_i}|^2 dV_i)^{\frac{1}{2}} (V(g_i))^{\frac{1}{2}}$$
$$\leq C_1  |s_i|^2_{\infty}  (\int_{M} |\nabla h_i|^2 dV_i)^{\frac{1}{2}} (V(g_i))^{\frac{1}{2}}$$
$$=2 C_1  |s_i|^2_{\infty}   (\int_{M} |X_i|^2 |\nabla |X_i||^2| dV_i)^{\frac{1}{2}} (V(g_i))^{\frac{1}{2}} $$
$$\leq 2 C_1  |s_i|^2_{\infty}  (\int_{M} |X_i|^2 |\nabla X_i|^2 dV_i)^{\frac{1}{2}} (V(g_i))^{\frac{1}{2}} $$
$$\leq  2 C_1 |s_i|^2_{\infty}  |X_i|_{\infty} (V(g_i))^{\frac{1}{2}}  (\int_{M} |\nabla X_i|^2 dV_i)^{\frac{1}{2}} $$
$$\leq \frac{ C_1 |s_i|^2_{\infty}   }{\sqrt{i}}  \int_{M} |X_i|^2 dV_i,$$
where $C_1$ is a positive constant depending on $n, p, \lambda_1, \lambda_2$.
\par
 It follows that
$$\frac{ \int_{M}  |X_i|^2 dV_i}{V(g_i)} \int_{M} |s_i|^2 dV_i \leq  \int_{M}  |X_i|^2 |s_i|^2 dV_i  +
\frac{ C_1 |s_i|^2_{\infty}   }{\sqrt{i}} \int_{M} |X_i|^2 dV_i.$$

\end{proof}

Now we prove Lemma \ref{co}.
By Theorem \ref{good2} and Lemma \ref{good1}, \ref{good3}, we get

$$\frac{ \int_{M}  t_i^2 |X_i|^2 dV_i}{V(g_i)} \int_{M} |s_i|^2 dV_i \leq  \int_{M} t_i^2 |X_i|^2 |s_i|^2 dV_i  +
\frac{ C_1 |s_i|^2_{\infty} }{\sqrt{i}}  \int_{M} t_i^2 |X_i|^2 dV_i$$
$$\leq \frac{C(n)}{\sqrt{i}} |s_i|_{\infty} (\int_{M} t_i^2 |X_i|^2 dV_i )^{\frac{1}{2}} (\int_{M}|s_i|^2 dV_i)^{\frac{1}{2}}+
\frac{ C_1 |s_i|^2_{\infty}   }{\sqrt{i}} \int_{M} t_i^2 |X_i|^2 dV_i$$
$$\leq \frac{C_2}{\sqrt{i}}  (\frac{\int_{M} t_i^2 |X_i|^2 dV_i}{V(g_i)} )^{\frac{1}{2}} \int_{M}|s_i|^2 dV_i+
\frac{ C_2}{\sqrt{i}}  \frac{\int_{M} t_i^2 |X_i|^2 dV_i}{V(g_i)} \int_{M}|s_i|^2 dV_i, $$
where $C_1, C_2$ are positive constants depending on $n, p, \lambda_1, \lambda_2$.

As $t_i=(\frac{V(g_i)}{\int_{M} |X_i|^2 dV_i})^{1/2}$, we see
\begin{equation} \label{29}
\frac{\int_{M} t_i^2 |X_i|^2 dV_i}{V(g_i)}=1.
\end{equation}

Then we see that for sufficiently large $i$,
$$ \int_{M} |s_i|^2 dV_i \leq  \frac{1}{2}\int_{M} |s_i|^2 dV_i.$$

\section{Appendix: Basic facts of elliptic genera}
\label{basic ell}
In this section we recall some basic facts about elliptic genera. Elliptic genera were first constructed by Ochanine \cite{Och} and Landweber-Stong \cite{LS88} in a topological way.
Witten gave a geometric interpretation to elliptic genera by showing that formally they are indices of Dirac operators on free loop space \cite{W87, W}. The theory of elliptic genera gives a connection among the Atiyah-Singer index theory, Kac-Moody affine Lie algebra, modular forms and quantum field theory. The background and introduction of elliptic genera can be found in \cite{HBJ, Lan}.
\par
Let $M$ be a $4n$ dimensional compact oriented manifold and $\{\pm
2\pi \sqrt{-1}z_{j},1\leq j\leq 2n\}$ denote the formal Chern roots of $T_{%
\mathbb{C}}M $, the complexification of the tangent vector bundle $TM$.

Let
$$\hat A(TM)=\prod_{j=1}^{2n}\frac{\pi\sqrt{-1}z_j}{\sinh(\pi\sqrt{-1}z_j)}, \ \ \ \hat L(TM)=\prod_{j=1}^{2n}\frac{2\pi\sqrt{-1}z_j}{\tanh(\pi\sqrt{-1}z_j)}$$
be the Hirzebruch $\hat A$-class and $\hat L$-class of $M$ respectively.

Let $E$ be a complex vector bundle and $\mathrm{ch}(E)$ the Chern character of $E$. For any complex number $t$, let
$$\Lambda_t(E)={\mathbb C}|_M+tE+t^2\Lambda^2(E)+\cdots ,
\\\ S_t(E)={\mathbb C}|_M+tE+t^2S^2(E)+\cdots$$  denote
the total exterior and symmetric powers  of $E$ respectively, which live in
$K(M)[[t]]$ (page 117-119 in \cite{At2}). The following relations on these two operations hold,
\be \label{SW} S_t(E)=\frac{1}{\Lambda_{-t}(E)},\ \ \ \
\Lambda_t(E-F)=\frac{\Lambda_t(E)}{\Lambda_t(F)}.
\ee

Denote $\widetilde{E}=E-\C^{\mathrm{rk} E}$ in $K(M)$.

The elliptic genera of $M$ can be defined as (chap. 6 in \cite{HBJ} and \cite{Liu95cmp})
\begin{equation*}
Ell_1(M)=\left\langle \widehat{L}(TM)\mathrm{ch}\left( \Theta \left( T_{\mathbb{C%
}}M\right)\otimes  \Theta_1 \left( T_{\mathbb{C%
}}M\right) \right) ,[M]\right\rangle \in \Q[[q]],
\end{equation*}%

\begin{equation*}
Ell_2(M)=\left\langle \widehat{A}(TM)\mathrm{ch}\left( \Theta \left( T_{\mathbb{C%
}}M\right)\otimes  \Theta_2 \left( T_{\mathbb{C%
}}M\right) \right) ,[M]\right\rangle \in \Q[[q^{\frac{1}{2}}]] ,
\end{equation*}%
where
\be   \label{Wibundles}
\Theta (T_{\mathbb{C}}M)=\overset{\infty }{\underset{j=1}{\otimes }}%
S_{q^{j}}(\widetilde{T_{\mathbb{C}}M}),\ \ \ \Theta_1(T_{\mathbb{C}}M)=\bigotimes_{j=1}^\infty
\Lambda_{q^j}(\widetilde{T_\C M}), \ \ \ \Theta_2(T_{\mathbb{C}}M)=\bigotimes_{j=1}^\infty \Lambda_{-q^{j-{1\over
2}}}(\widetilde{T_{\mathbb C}M})
\ee
are the Witten bundles introduced in \cite{W}. One can expand these elements into Fourier series,
\be \Theta \left( T_{\mathbb{C%
}}M\right)\otimes  \Theta_1 \left( T_{\mathbb{C%
}}M\right)=A_0(T_\CC M)+A_1(T_\CC
M)q+\cdots=\CC+2(T_\CC M-\CC^{4n})q+\cdots,\ee
\be \label {expand B} \Theta \left( T_{\mathbb{C%
}}M\right)\otimes  \Theta_2 \left( T_{\mathbb{C%
}}M\right)=B_0(T_\CC M)+B_1(T_\CC M)q^{1\over2}+\cdots=\CC-(T_\CC M-\CC^{4n})q^{1\over 2}+\cdots.
\ee
Hence we have
\be  Ell_1(M)=\left\langle \widehat{L}(TM),[M]\right\rangle+2\left\langle \widehat{L}(TM)\mathrm{ch}\left(T_\CC M-\CC^{4n} \right) ,[M]\right\rangle q+\cdots,  \ee
\be \label{expand ell2} Ell_2(M)=\left\langle \widehat{A}(TM),[M]\right\rangle-\left\langle \widehat{A}(TM)\mathrm{ch}\left(T_\CC M-\CC^{4n} \right) ,[M]\right\rangle
q^{1\over 2}+\cdots\ee
and see that $Ell_1(M)$ is a $q$-deformation of $\sigma(M)$, the signature of $M$; and $Ell_2(M)$ is a $q$-deformation of $\hat A (M)$, the $\hat A$ genus of $M$.

By the Atiyah-Singer index theorem \cite{AS},
$Ell_1(M)$ can be expressed analytically as index of the twisted signature operator
\be \label{ell1ana} Ell_1(M)= \ind(d_s\otimes(\Theta \left( T_{\mathbb{C%
}}M\right)\otimes\Theta_1 \left( T_{\mathbb{C%
}}M\right)))\in \Z[[q]],\ee
where $d_s$ is the signature operator;
and when $M$ is spin,
$Ell_2(M)$ can be expressed analytically as index of the twisted Dirac operator,
\be \label{ell2ana} Ell_2(M)= \ind(D\otimes (\Theta \left( T_{\mathbb{C%
}}M\right)\otimes\Theta_2 \left( T_{\mathbb{C%
}}M\right)))\in \Z[[q^{1/2}]],\ee
where $D$ is the Atiyah-Singer spin Dirac operator on $M$ \cite{W}.

One of the important properties of elliptic genera is modularity. More precisely, we have the following two theorems.
\begin{thm} \label{mod 1}
 $Ell_1(M)$ and $Ell_2(M)$ are modularly related as
\be Ell_1(M, -1/\tau)={(2\tau)}^{2n}Ell_2(M, \tau). \ee
\end{thm}

\begin{proof}
See page 119-120 in \cite{HBJ} and \cite{Liu95cmp}.
\end{proof}

\begin{thm} \label{mod 2}
(i) $\forall k\geq 0$,  the $B_k(T_\CC M)$ in the expansion (\ref{expand B}) is a virtual bundle, which is an integral linear combination of bundles of type
$$S^{i_1}(T_\CC M)\otimes \cdots \otimes S^{i_r}(T_\CC M)\otimes \Lambda^{j_1}(T_\CC M)\otimes \cdots \otimes \Lambda^{j_s}(T_\CC M),$$ who are subbundles of
tensor products of $T_\CC M$ of power at most $k$; \newline
(ii) $Ell_2(M)$ is determined by $ \ind (D\otimes B_k(T_\C M)), \, 0 \leq k \leq \left[\frac{n}{2}\right].$
\end{thm}
\begin{proof} The first statement can be simply observed from (\ref{SW}), (\ref{Wibundles})  and (\ref{expand B}).

The proof of second statement can be found in Section 8.2 in \cite{HBJ} and \cite{Liu95cmp}.
We recap here to show how the elliptic genus is determined by $B_k(T_\CC M)$ more explicitly.

Let $$ SL_2(\Z):= \left\{\left.\left(\begin{array}{cc}
                                      a&b\\
                                      c&d
                                     \end{array}\right)\right|a,b,c,d\in\Z,\ ad-bc=1
                                     \right\}
                                     $$
 as usual be the famous modular group. Let
$$S=\left(\begin{array}{cc}
      0&-1\\
      1&0
\end{array}\right), \ \ \  T=\left(\begin{array}{cc}
      1&1\\
      0&1
\end{array}\right)$$
be the two generators of $ SL_2(\Z)$. Their actions on
$\mathbf{H}$ are given by
$$ S:\tau\rightarrow-\frac{1}{\tau}, \ \ \ T:\tau\rightarrow\tau+1.$$

Let
$$ \Gamma_0(2)=\left\{\left.\left(\begin{array}{cc}
a&b\\
c&d
\end{array}\right)\in SL_2(\Z)\right|c\equiv0\ \ (\rm mod \ \ 2)\right\},$$

$$ \Gamma^0(2)=\left\{\left.\left(\begin{array}{cc}
a&b\\
c&d
\end{array}\right)\in SL_2(\Z)\right|b\equiv0\ \ (\rm mod \ \ 2)\right\}$$
be the two modular subgroup of $SL_2(\Z)$. It is known
that the generators of $\Gamma_0(2)$ are $T,ST^2ST$, while the
generators of $\Gamma^0(2)$ are $STS,T^2STS$ \cite{Cha}. It can be shown that $Ell_1(M)$ is a modular form of weight $2n$ over $\Gamma_0(2)$ and $Ell_2(M)$ is a modular
form of weight $2n$ over $\Gamma^0(2)$  \cite{Liu95cmp}.

If $\Gamma$ is a modular subgroup, let
$\mathcal{M}_\R(\Gamma)$ denote the ring of modular
forms over $\Gamma$ with real Fourier coefficients.  We introduce four
explicit modular forms (page 119 in \cite{HBJ}),
$$ \delta_1(\tau)=\frac{1}{4}+6\sum_{n=1}^{\infty}\underset{d\ odd}{\underset{d|n}{\sum}}dq^n, \ \ \ \
\varepsilon_1(\tau)=\frac{1}{16}+\sum_{n=1}^{\infty}\underset{d|n}{\sum}(-1)^dd^3q^n\ ,$$
$$\delta_2(\tau)=-\frac{1}{8}-3\sum_{n=1}^{\infty}\underset{d\ odd}{\underset{d|n}{\sum}}dq^{n/2}, \ \ \ \
\varepsilon_2(\tau)=\sum_{n=1}^{\infty}\underset{n/d\ odd}{\underset{d|n}{\sum}}d^3q^{n/2}\ .$$ They have
the following Fourier expansions in $q^{1/2}$:
$$\delta_1(\tau)={1\over 4}+6q+6q^2+\cdots,\ \ \ \ \varepsilon_1(\tau)={1\over
16}-q+7q^2+\cdots\ , $$
$$\delta_2(\tau)=-{1\over 8}-3q^{1/2}-3q+\cdots,\ \ \ \
\varepsilon_2(\tau)=q^{1/2}+8q+\cdots\ .$$ where the
\textquotedblleft $\cdots$" terms are the higher degree terms, all
of which have integral coefficients. They also satisfy the
transformation laws, \be
\delta_2\left(-\frac{1}{\tau}\right)=\tau^2\delta_1(\tau),\ \ \ \ \
\ \ \ \ \
\varepsilon_2\left(-\frac{1}{\tau}\right)=\tau^4\varepsilon_1(\tau).\ee
One has that $\delta_1(\tau)\ (resp.\ \varepsilon_1(\tau) ) $
is a modular form of weight $2 \ (resp.\ 4)$ over $\Gamma_0(2)$,
while $\delta_2(\tau) \ (resp.\ \varepsilon_2(\tau))$ is a modular
form of weight $2\ (resp.\ 4)$ over $\Gamma^0(2)$, and moreover
$\mathcal{M}_\R(\Gamma^0(2))=\R[\delta_2(\tau),
\varepsilon_2(\tau)]$.

Therefore one can express $Ell_2(M)$ in terms of $8\delta_2(\tau)$ and $\varepsilon_2(\tau)$ as
\be \label{baseexpand}
Ell_2(M)
 =h_0(8\delta_2(\tau))^n+h_1(8\delta_2(\tau))^{n-2}\varepsilon_2(\tau)
+\cdots+h_{[\frac{n}{2}]}(8\delta_2(\tau))^{\bar n}\varepsilon_2(\tau)^{[\frac{n}{2}]} ,\ee where
$\bar n=0$ if $n$ is even and $\bar n=1$ if $n$ is odd, and each $h_r$, $0\leq r\leq [\frac{n}{2}] $, is an integer.
They are all indices of certain twisted Dirac operators on
$M$. Write $\Theta \left( T_{\mathbb{C%
}}M\right)\otimes  \Theta_2 \left( T_{\mathbb{C%
}}M\right)$ as
\be \Theta \left( T_{\mathbb{C%
}}M\right)\otimes  \Theta_2 \left( T_{\mathbb{C%
}}M\right)=B_0(T_\C M)+B_1(T_\C
M)q^{1\over2}+\cdots. \ee The $B_i$'s carry canonically induced
Hermitian metrics and connections from the Riemannian metric and Levi-Civita connection on $TM$. Then
\be \label{qexpand} Ell_2(M)=\ind (D\otimes B_0(T_\C M))+\ind(D\otimes B_1(T_\C
M))q^{1\over2}+\cdots. \ee

Comparing the $q$-coefficients in (\ref{baseexpand}) and (\ref{qexpand}) and noticing that that $8\delta_2(\tau)$ starts from $-1$, one sees that each $h_r$ is a canonical
linear combination of $\ind(D\otimes B_j(T_\C M)), 0\leq j \leq r.$ So
 we see that $Ell_2(M)$ is determined by $B_{k}(T_\C M), 0\leq k\leq [\frac{n}{2}].$
 \end{proof}

The following Corollary is an easy consequence of Theorem \ref{mod 1} and \ref{mod 2}.
\begin{cor}
If $Ell_1(M) \neq 0$ or $\ Ell_2(M) \neq 0$, then $ \ind (D\otimes B_k(T_\C M)) \neq 0$ for some $k \leq \left[\frac{n}{2}\right]$.
\end{cor}

Anther important property of elliptic genus is rigidity.
\begin{thm}[\protect Witten-Bott-Taubes-Liu Rigidity] \label{rig} The Witten operators
$$d_s\otimes(\Theta \left( T_{\mathbb{C%
}}M\right)\otimes\Theta_1 \left( T_{\mathbb{C%
}}M\right)), \ \  D\otimes (\Theta \left( T_{\mathbb{C%
}}M\right)\otimes\Theta_2 \left( T_{\mathbb{C%
}}M\right))$$ are rigid.
\end{thm}
\begin{proof}
See \cite{ BT, Liu95, Liu96, Tau}.
\end{proof}

Recently Vernge \cite{Ver} reproved Witten rigidity by developing the theory of  elliptic bouquet.

\par
Our Theorem \ref{mm2} gives a relationship between curvature and elliptic genera. The following example shows that on a closed spin Riemannian manifold, without the
curvature assumptions in Theorem \ref{mm2}, even if the isometry group is infinite, the elliptic genera do not necessarily vanish.
\begin{exm} \label{va}
Let $M=M(5; 2)$ be a smooth quadric hypersurface in $\mathbb{CP}^5$. This is a 8 dimensional closed spin manifold carrying the linear $SO(6)$ action and therefore a
nontrivial $S^1$-action, preserving the K\"ahler metric on $M$ induced by the embedding $ M \subset \C P^5$. Then $\hat A(M)=0$
by the famous Atiyah-Hirzebruch vanishing theorem \cite{AH70}. We will show that $\int_M \hat
A(TM)\mathrm{ch}(T_{\mathbb{C}} M)\neq 0$, which implies $Ell_2(M)\neq 0$ by (\ref{expand ell2}).  Actually by the 8 dimensional miraculous cancellation formula
(\cite{Liu95cmp}), one has
$$\sigma(M)=24\hat A(M)-\int_M \hat A(TM)\mathrm{ch} (T_{\mathbb{C}} M),$$
where $\sigma(M)$ is the signature. Since $\hat A(M)=0$, we just need to show that $\sigma(M)\neq 0$. Let $x\in H^2(\mathbb{CP}^5, \mathbb{Z})$ be the generator. Then by the
Hirzebruch signature theorem and Poincar\'e duality, one sees that
$$\sigma(M)=\left\langle \left(\frac{x}{\tan x} \right)^6 \tan(2x), [\mathbb{CP}^5]\right\rangle=\mathrm{Res}_{x=0} \left(\frac{\tan 2x}{(\tan x)^6}\right)=2. $$
\end{exm}

Closely related to the elliptic genera is the Witten genus
\begin{equation*}
W(M)=\left\langle \widehat{A}(TM)\mathrm{ch}\left( \Theta \left( T_{\mathbb{C%
}}M\right) \right) ,[M]\right\rangle \in \Q[[q]].
\end{equation*}
When $M$ is spin,
$$W(M)=\ind(D\otimes \Theta \left( T_{\mathbb{C%
}}M\right))\in \mathbb{Z}[[q]].$$

The Witten genus is conjectured to be an obstruction to positive Ricci curvature on string manifolds. More precisely, the famous Stolz conjecture \cite{St} says that if $M$ is a smooth closed string manifold of dimension $4n$ and  admits a Riemannian metric with positive Ricci curvature, then the Witten genus $W(M)$ vanishes. This conjecture can be viewed as the higher version of the classical Lichnerowicz theorem \cite{Li}. So far the Stolz conjecture is still open.

The Witten genus is also an obstruction to simply connected Lie group actions on string manifolds. Actually it has been shown that a string manifold with a nontrivial $S^3$-action has vanishing Witten genus \cite{Des94, Liu95}. Nevertheless, on spin manifolds, the Witten genus is not rigid.

\end{document}